\documentclass[11pt,A4paper]{article}

\usepackage[left=32mm,right=32mm,top=30mm,bottom=32mm]{geometry}
\usepackage{labelfig}
\usepackage{epsfig}
\usepackage{epstopdf}

\usepackage{xfrac}

\usepackage[percent]{overpic}

\usepackage{color}
\usepackage{amsthm,amsmath,amssymb}
\usepackage{booktabs}
\usepackage{mathpazo}
\usepackage{microtype}
\usepackage{overpic}
\usepackage{bm}
\usepackage{sectsty}
\usepackage[
	pdftitle={PDFTitle},
	pdfauthor={Hugo Parlier},
	ocgcolorlinks,
	linkcolor=linkred,
	citecolor=linkred,
	urlcolor=linkblue]
{hyperref}

\definecolor{linkred}{rgb}{0.48,0.1,0.05}
\definecolor{linkblue}{RGB}{16, 78, 139}

\usepackage[hang,flushmargin]{footmisc}
\usepackage{enumitem}
\usepackage{titlesec}
	\titlespacing{\section}{0pt}{12pt}{0pt}
	\titlespacing{\subsection}{0pt}{6pt}{0pt}
	
\titlelabel{\thetitle.\quad}

\makeatletter 

\long\def\@footnotetext#1{%
\H@@footnotetext{%
\ifHy@nesting 
\hyper@@anchor{\@currentHref}{#1}%
\else 
\Hy@raisedlink{\hyper@@anchor{\@currentHref}{\relax}}#1%
\fi 
}}

\def\@footnotemark{%
\leavevmode 
\ifhmode\edef\@x@sf{\the\spacefactor}\nobreak\fi 
\H@refstepcounter{Hfootnote}%
\hyper@makecurrent{Hfootnote}%
\hyper@linkstart{link}{\@currentHref}%
\@makefnmark 
\hyper@linkend 
\ifhmode\spacefactor\@x@sf\fi 
\relax 
}%

\ifFN@multiplefootnote%
\renewcommand*\@footnotemark{%
\leavevmode 
\ifhmode 
\edef\@x@sf{\the\spacefactor}%
\FN@mf@check 
\nobreak 
\fi 
\H@refstepcounter{Hfootnote}%
\hyper@makecurrent{Hfootnote}%
\hyper@linkstart{link}{\@currentHref}%
\@makefnmark 
\hyper@linkend 
\ifFN@pp@towrite 
\FN@pp@writetemp 
\FN@pp@towritefalse 
\fi 
\FN@mf@prepare 
\ifhmode\spacefactor\@x@sf\fi 
\relax%
}%
\fi 

\makeatother 

\theoremstyle{plain}
\newtheorem{theorem}{Theorem}[section]

\newtheorem{lemma}[theorem]{Lemma}
\newtheorem{corollary}[theorem]{Corollary}

\theoremstyle{definition}

\newtheorem{remark}[theorem]{Remark}

\newcommand{\N}{{\mathbb N}}

\newcommand{\Z}{{\mathbb Z}}
\newcommand{\C}{{\mathbb C}}

\newcommand{\CC}{{\mathcal C}}

\newcommand{\G}{{\mathcal G}}

\newcommand{\rank}{{\rm rank}}

\sectionfont{\large \sc}
\subsectionfont{\normalsize}

\long\def\symbolfootnote[#1]#2{\begingroup%
\def\thefootnote{\fnsymbol{footnote}}\footnote[#1]{#2}\endgroup}

\def\blfootnote{\xdef\@thefnmark{}\@footnotetext}

\begin{document}

{\Large \bfseries \sc Curve graphs on surfaces of infinite type}

{\bfseries Ariadna Fossas\symbolfootnote[1]{\normalsize Research supported by ERC grant agreement number 267635 - RIGIDITY} and Hugo Parlier\symbolfootnote[2]{\normalsize Research supported by Swiss National Science Foundation grant number PP00P2\textunderscore 128557
{\em 2010 Mathematics Subject Classification:} Primary: 57M50. Secondary: 05C63, 57M15. \\
{\em Key words and phrases:} curve complexes, infinite type surfaces, hyperbolic surfaces}
}

{\em Abstract.} We define and study analogs of curve graphs for infinite type surfaces. Our definitions use the geometry of a fixed surface and vertices of our graphs are infinite multicurves which are bounded in both a geometric and a topological sense. We show that the graphs we construct are generally connected, infinite diameter and infinite rank. 
\vspace{1cm}

\section{Introduction} \label{s:introduction}

The curve complex, the pants complex and a number of other simplicial complexes and graphs related to simple closed curves on finite type surfaces have been used in multiple contexts for the study of the Teichm\"uller and moduli space, mapping class groups and related topics. In particular the geometry of these complexes has played a part in both understanding the geometries of Teichm\"uller spaces and a geometric group theory approach to the study of the mapping class group. 

The Teichm\"uller theory of infinite type surfaces is not nearly as developed as the finite type case, but there have been a number of interesting results about geometric properties of such surfaces (see for instance \cite{Basmajian}) and their deformation spaces (see 
\cite{Basmajian-Kim,AlessandriniEtAl,AlessandriniEtAl2}). 

There are also recent results about simplicial complexes related to infinite type surfaces. The usual curve graph can of course be defined on an infinite type surface and for instance it is a result of Hernandez and Valdez that the mapping class group is the automorphism group of this graph under certain non-trivial conditions. From the coarse geometric point of view, as a metric space it is not particularly exciting as it has diameter $2$. In another direction, there is a recent result of Bavard \cite{Bavard} about a ray graph associated to infinite type planar surfaces that has infinite diameter and is Gromov hyperbolic. 

Our goal is to contribute to this setting by defining and studying another graph associated to infinite type surfaces. Roughly speaking vertices are multicurves whose complement has bounded complexity and relating vertices if they can be realized disjointly. Depending on how one makes the about sentence precise, the graph in question is generally disconnected.

We define our graphs with respect to fixed hyperbolic structure on a surface. In the case of finite type surfaces, this is equivalent to the usual setup, but in the case of infinite type surfaces, this makes a big different. To be precise, we require that our surfaces have a certain type of bounded geometry, namely that they admit a (generalized) pants decomposition where supremum of the lengths of the individual curves in the decomposition is bounded. Deformation spaces of such surfaces have been studied by Alessandrini et al. and have been called upper-bounded surfaces. 

We fix a surface $M$ with a hyperbolic metric as above. Now for each $K\in \N \cup \{\infty\}$, we get a graph $\G_K(M)$ where vertices are multicurves that also have this bounded property length and whose complementary regions have complexity at most $K$. Again, edges appear when the multicurves can be realized disjointly (the cases $K=0, \infty$ are special - see the next section for the precise definitions). Note that for finite type surfaces, our graphs for certain $K$ are essentially the curve graph, the pants graph and some sort of set of graphs "in between". In particular they are all connected and generally have interesting geometries.

It is not a priori obvious that these graphs are connected in the infinite type case (in fact without the bounded length property they aren't necessarily) so our first theorem is about the connectedness.

\begin{theorem}
For any upper-bounded $M$ and any $K \in \N \cup \{\infty\}$, the graph $\G_K(M)$ is connected.
\end{theorem}

Again, from a geometric point of view, we don't want the graphs to be {\it too} connected - by which we mean finite diameter. For finite $K$ we show they aren't - and they are as far as possible from being Gromov hyperbolic. 

\begin{theorem}
For any upper-bounded $M$ of infinite type and $K\in \N$, $\rank(\G_K(M))= \infty$.
\end{theorem}

To show this we exhibit arbitrarily large quasi-flats via subsurface projections onto finite type subsurface curve graphs. We note that we don't know whether these graphs are quasi-isometric for different $K$. 

We conclude the paper with the example of a graph $G_\infty(Z)$, for a particular type of surface $Z$, which has finite diameter (at most $3$). It's not completely obvious that the diameter of this graph is bounded (it is proved in Theorem \ref{thm:Z}). The example is intriguing because it is in some sense the limit of infinite rank metric spaces.

\section{Preliminaries}\label{s:setup}

We begin by defining the graphs we will be studying. Let $M$ be an orientable hyperbolic surface with non-trivial fundamental group. In general $M$ will be a surface of infinite type but it is interesting to note that many of the definitions apply to finite type surfaces and give rise to some of the usual combinatorial graphs associated to curves on surfaces. It's important to note that, in contrast to the usual setting of curve graphs, we consider a fixed hyperbolic structure on $M$. We make the following further assumption on $M$: we assume that $M$ admits a geodesic pants decomposition such that the supremum of the lengths of the individual curves is finite. This condition on the metric is referred to as being {\it upper bounded} in \cite{AlessandriniEtAl} where the authors define and study Fenchel-Nielsen coordinates for this type of surface. It might be worth remarking - but we will not dwell on it here - that the only real requirement we need for any of our results is that we have a metric surface which is bi-lipschitz equivalent to a hyperbolic surface as described above.

When $M$ is of finite type, the complexity $\kappa(M)$ of $M$ is the number of curves in a pants decomposition of $M$. So if $M$ is homeomorphic to a surface of genus $g$ with $n$ boundary curves, then $\kappa(M) = 3g-3 + n$. 

We recall the definition of the usual curve graph $\C(M)$: vertices are isotopy classes of non-trivial simple closed curves and two vertices share an edge if they can be realized disjointly on $M$. In this context, when $M$ has boundary, non-trivial means non-isotopic to a point and non-peripheral to boundary. We think of $\C(M)$ as a geometric graph where edges have length $1$. Note that on a surface of infinite type, any two curves are distance at most $2$ in this graph. Its geometry - from this point of view - is somewhat limited. 

For general $M$ as above, and for an integer $K>0$, we define the following graph $\G_K(M)$:

- Vertices of $\G_K(M)$ are (geodesic) multicurves $\mu$ of $M$ such that any connected component $\Gamma$ of $M \setminus \mu$ satisfies $\kappa(\Gamma) \leq K$ (finite complexity condition) and
$$
\sup\{ \ell(\alpha_\mu) \,|\, \alpha_\mu \in \mu \mbox{ is a connected component of }\mu\}  < +\infty \mbox{ (finite length condition).}
$$

- Two vertices $\mu$ and $\mu'$ span an edge if they can be realized disjointly. 

We note that if $M$ is of finite type and $K = \kappa(M) -1$ then $\G_K(M)$ is the usual curve graph with extra vertices corresponding to all simplices.

For $K=0$ we define $\G_0(M)$ similarly. The vertex set is defined as previously (so in this case geodesic pants decompositions with finite supremum of individual curve lengths ) but the edge set is slightly different. 

For this we recall the definition of an {\it elementary move} between pants decompositions. Two pants decompositions $\mu, \mu'$ are related by an elementary move if they differ by exactly one curve and if the curves that distinguish them intersect minimally on the complexity $1$ subsurface which they share (see Figure \ref{fig:elementarymoves}).

\begin{figure}[htbp]
\begin{center}
\includegraphics{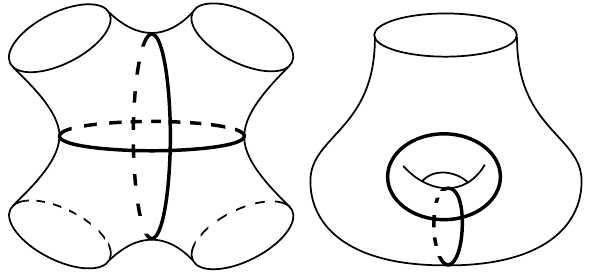}
\caption{The two types of elementary moves}
\label{fig:elementarymoves}
\end{center}
\end{figure}

Elementary moves can be performed {\it simultaneously} if they are performed on disjoint complexity $1$ subsurfaces. Two pants decompositions in $\G_0(M)$ share an edge if they differ by (a possibly infinite number) of simultaneous elementary moves. This graph is sometimes referred to as the {\it diagonal pants graph}.

For $K=\infty$, we define a graph $\G_\infty(M)$ as follows. Vertices are (geodesic) multicurves $\mu$ of $M$ such that 
$$
\sup \{ \kappa(\Gamma) \,|\, \Gamma \mbox{ is a connected component of }M\setminus \mu \}< \infty$$
and
$$
\sup\{ \ell(\alpha_\mu) \,|\, \alpha_\mu \in \mu \mbox{ is a connected component of }\mu\}  < +\infty.
$$

We'll be thinking of these graphs as metric graphs where edge lengths are all $1$ and we'll be interested in their geometry. 

{\bf Observation.} We point out that, by definition, if $0<K'<K$ then
$$
\G_{K'}(M) \subset \G_K(M).
$$
The vertices of $\G_0(M)$ lie in all $\G_K(M)$ but as any two elements of $\G_0(M)$ intersect, none of the edges of $\G_0(M)$ lie in $\G_K(M)$ for $K>0$. However, any pants decompositions that share an edge in $\G_0(M)$ differ on the complement of a multicurve where each connected component has complexity at most one. This means that they are at distance $2$ in any of the graphs $\G_K(M)$ for $K>0$. The converse is not (necessarily) true.

The first step will be to show that these graphs are connected and the previous observation will be crucial in showing that.

\section{Connectedness}\label{s:connectedness}

In this section we prove that the graphs we defined are all connected. 

We begin by showing that for any $\mu \in \G_K(M)$ there exists a pants decomposition $\mu_P \in  \G_K(M)$ that contains $\mu$. It is immediate that this is true when $M$ is of finite type. As such it is also immediate that there exists a geodesic pants decomposition which contains $\mu$ for any type of $M$ - what requires a proof is that one can choose this pants decomposition to lie in $\G_K(M)$ because if one chooses an arbitrary pants decomposition then it won't necessarily satisfy the finite length condition. We state the result as a lemma.

\begin{lemma}
For any $\mu \in \G_K(M)$ there exists a pants decomposition $\mu_P \in  \G_K(M)$ that contains $\mu$.
\end{lemma}

\begin{proof}
This essentially follows from results on the length of pants decompositions. In particular, any surface of area $A$ and boundary length at most $B$ admits a pants decomposition of length at most a function of $A$ and $B$ (this follows from generalizations of results of Bers and Buser, see for instance \cite{BalacheffParlierSabourau}). 

Now let $L:= \sup\{ \ell(\alpha_\mu) \,|\, \alpha_\mu \in \mu \mbox{ is a connected component of }\mu\} $. Let $M'$ be a connected component of $M\setminus \mu$. As it is of complexity at most $K$, it has at most $K+2$ boundary curves. Each is of length at most $L$ so 
$$
\ell(\partial M') \leq (K+2) L.
$$
As $M'$ is hyperbolic, its area is also bounded above by a function of $K$. As such, by the result described above, $M'$ admits a pants decomposition where every curve has length bounded above by a function of $K$ and $L$. As this is true for all connected $M' \subset M \setminus \mu$, we obtain a pants decomposition of $M$ which contains $\mu$ and continues to enjoy the finite length property.
\end{proof}

We now prove the following theorem.

\begin{theorem}
For any $K \in \N \cup \{\infty\}$, the graph $\G_K(M)$ is connected.
\end{theorem}

\begin{proof}

In light of the observation at the end of the preceding section, it suffices to prove that $\G_0(M)$ is connected. 

We begin by proving the following claim.

{\it Claim 1.} For any $v,w \in \G_0(M)$, there exists $N_{v,w}$ such that all curves $\alpha \in v$ and $\beta \in w$ satisfy
$$
i(\alpha, w) \leq N_{v,w} \mbox{ and } i(\beta, v) \leq N_{v,w}.
$$

To prove the claim, observe that the lengths of curves in $v$ and $w$ are uniformly bounded by a constant, say $L$. Each intersection point between a curve $\alpha \in v$ and $w$ forces $\alpha$ to enter the collar of a curve in $\beta$ and then leave again. By the collar lemma, the width of this collar is uniformly bounded below by a positive constant $C_L$ that only depends on $L$. As such, if $\alpha$ intersects $w$ at least $k$ times, its length must be greater than $k C_L.$ As such it follows that $k$ satisfies
$$
k < \frac{L}{C_L}.
$$
A symmetric argument works for the intersection between $\beta\in w$ and $v$ and this proves the claim. 

The key to the argument is the following claim.

{\it Claim 2.} There exists a positive function $F: \N \to \N$ such that if $v,w \in \G_0(M)$ and for all $\alpha\in v$ and all $\beta\in w$ satisfying 
$$i(\alpha, w) \leq N \mbox{ and } i(\beta, v) \leq N$$
then 
$$
d (v,w) \leq F(N)
$$
where $d$ denotes distance in $\G_0(M)$.

This is Lemma 4.4 from \cite{AndersonParlierPettet} which applies to both finite type and infinite type surfaces.

The result is a simple consequence of the two claims. 
\end{proof}

\section{Quasi-convexity of strata and infinite rank}

In this section we prove that for $K\in \N$, strata in $\G_K(M)$ corresponding to the set of multicurves containing a fixed multicurve $\mu$, are quasi-convex. Using this we are able to deduce that for all infinite type surfaces $M$, and all $K\in \N$, the graph $\G_K(M)$ is of infinite rank.

We begin with the following lemma. For convenience, we denote distance in $\G_K(M)$ by $d$.

\begin{lemma}\label{lem:proj}
Let $M'$ be a subsurface of $M$ of complexity $K'>K$. We denote by $\CC(M')$ the usual curve graph associated to $M'$. Then there exists a projection
$$
\pi_{M'} : \G_K(M) \to \CC(M')
$$
satisfying
$$
d_{\CC(M')} (\pi_{M'}(v),\pi_{M'}(w)) \leq 9 \, d(v,w).
$$
\end{lemma}

\begin{proof}
As $K'>K$, any $v\in \G_K(M)$ contains a curve $\alpha$ such that $\alpha \cap M' \neq \emptyset.$
Thus the following map is well defined: for any $v\in  \G_K(M)$ we define $\pi_{M'}(v)$ to be any single curve in the subsurface projection of $v$ to $M'$.

(Recall that a subsurface projection to $M'$ is the collection of isotopy classes of simple closed curves formed by an $\varepsilon$-neighborhood of $\{v\cap M'\} \cup \partial M'$.)

We observe that if any two curves $\alpha,\beta$ lie in the subsurface projection of the same multiarc $v$, then 
$$
i(\alpha,\beta) \leq 4.
$$

We state the following well known fact about the curve graph which can be shown by a cut and paste type argument (see for example \cite{Bowditch}).

{\it Fact.} Any two curves on a surface which intersect at most $k$ times are distance at most $2 \log_2(k)+2$ apart in the underlying curve complex. 

Suppose that $K>0$ and let $v$ and $w$ be vertices of $ \G_K(M)$ joined by an edge. Both $\pi_{M'}(v)$ and $\pi_{M'}(w)$ lie in the subsurface projection of the multicurve $v\cup w$ so by the above are distance at most $4$ in $\CC(M')$. The result then follows by induction.

Now if $K=0$, we argue a little bit differently. Let $v$ and $w$ share an edge in $ \G_0(M)$. If $a$ is an arc (or a curve) in $v\cap M'$ and $b$ an arc (or a curve) in $w\cap M'$ then
$$
i(a,b) \leq 2.
$$
Indeed both $a$ and $b$ are subsets of curves of $v$ and $w$ and any two curves in $v$ and $w$ intersect at most $2$ times. As a consequence if $\alpha= \pi_{M'}(v)$ and $\beta=\pi_{M'}(w)$
then 
$$
i(\alpha,\beta) \leq 12
$$
(each end of an arc can produce $2$ intersection points in the projection and each arc intersection point can produce $4$ in the projection). We deduce that $\alpha$ and $\beta$ are distance at most $9$ in curve graph of $M'$. Again the result follows from induction.
\end{proof}

\begin{remark} In the above proof we obtain a better bound (namely $6\, d(v,w)$) in the case of $K>0$ than in the case of $K=0$. It might be interesting to know by how much these constants can be improved. 
\end{remark}

An immediate consequence of the above lemma is the following.

\begin{corollary}
Let $M' \subset M$ be a subsurface of complexity $K'$. Then $\CC(M')$ (uniformly) quasi-isometrically embeds into $\G_K(M)$ for $K=K'-1$. 
\end{corollary}
\begin{proof}
Let $\mu$ be a geodesic multicurve so that $M' $ is the only connected component of positive complexity of $M\setminus \mu$ (and such that $\mu \in \G_{K'}(M)$). Associated to $\mu$ is a natural map  
$$
\sigma_\mu: \CC(M') \to \G_{K}(M)
$$
defined as
$$
\sigma_\mu(\alpha) =  \alpha \cup \mu.
$$
By the above lemma the map $\sigma_\mu$ is a quasi-isometric embedding. 
\end{proof}

We observe that this implies that for infinite type $M$, all of the graphs $\G_K(M)$ for $K\neq \infty$ are of infinite diameter. We can now show that they also have infinite rank.

\begin{theorem}
For any infinite type $M$ and any $K\in \N$, $\rank(\G_K(M))= \infty$.
\end{theorem}

\begin{proof}
As $M$ is of infinite type, for any $K$ we can find an infinite set of subsurfaces $M_i, \, i \in \N^{*}$ such that $M_i \cap M_j = \emptyset $ for $i\neq j$ and $\kappa(M_i) = K+1$ for all $i\in \N^{*}$.

Now for any $n\in \N^{*}$, we consider bi-infinite geodesics $\gamma_i$ on each $\CC(M_i)$, $i=1,\hdots,n$. A choice of origin on each of them and a direction gives us a natural embedding of the set of points of $\Z^n$ into $\G_K(M)$ as we shall explain. This embedding will be a quasi-isometry as in the previous corollary. 

We consider the following $\ell_\infty$ metric on the product $P_n$ of these curve graphs. More precisely, let
$$P_n:=\Pi_{1\leq i \leq n} \CC(M_i)$$
be endowed with the following metric: two elements $(\alpha_1,\hdots,\alpha_n)$, $(\alpha_1,\hdots,\alpha_n)$ are at distance $1$ if 
$$
\max_{i=1,\hdots,n} d_{\CC(M_i)}(\alpha_i,\beta_i) = 1.
$$
In this metric space, the restriction to the metric on the embedding of $\Z^n$ is the $\ell_\infty$ metric, but is naturally quasi-isometric to the usual metric on $\Z^n$. As such we have a quasi-isometric embedding of $\Z^n$ in $P_n$.

Using Lemma \ref{lem:proj}, we now get a distance (quasi) non increasing map from $\G_K(M)$ to $P_n$. As in the previous corollary, by choosing a pants decomposition on the complementary region of the $M_i$, $i=1,\hdots,n$, we can quasi-embed $P_n$ into $\G_K(M)$. In turn this provides the quasi-isometric embedding of $P_n$ - and thus of $\Z^n$ - we were looking for. As this can be done for any $n$, the theorem is proved.
\end{proof}

\section{A case of finite diameter}

We conclude the article by studying a particular example of infinite type surface $Z$ and show that for this surface $\G_\infty(Z)$ has diameter at most $3$. The reason the example is intriguing is that $\G_\infty(Z)$ is in some sense the limit of infinite rank metric spaces. 

We describe the surface $Z$ in terms of Fenchel-Nielsen coordinates. Consider the infinite cubic graph as in Figure \ref{fig:graph}.

\begin{figure}[htbp]
\begin{center}
\includegraphics[width=14cm]{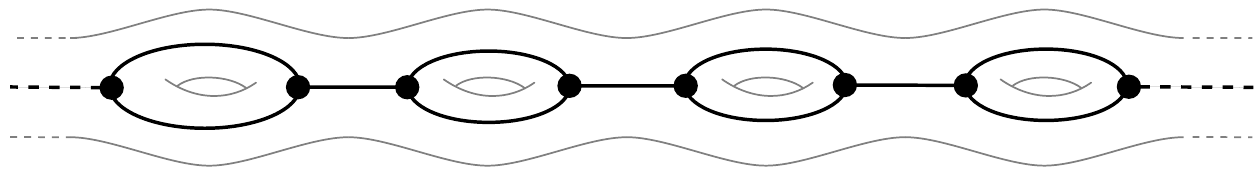}
\caption{An infinite cubic graph and the surface $Z$ lurking behind}
\label{fig:graph}
\end{center}
\end{figure}

This is the graph dual to a pants decomposition (vertices correspond to pants and edges to pants curves). The surface locally looks like Figures \ref{fig:graph} and  \ref{fig:bigsurface}. We now construct $M$ by taking each pair of pants to have all three lengths equal to a fixed constant $\ell_0$ and twist parameters equal to $0$.

\begin{figure}[htbp]
\begin{center}
\includegraphics[width=14cm]{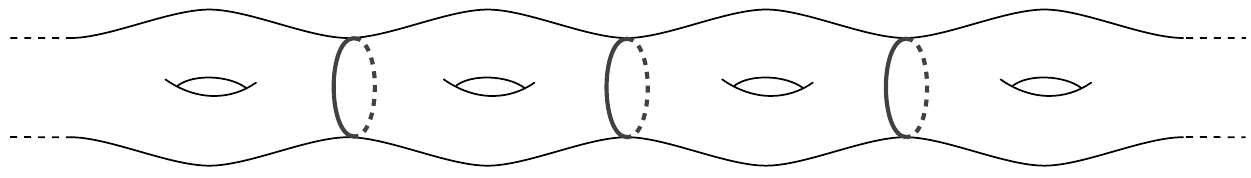}
\caption{The surface $Z$ with the curves $\gamma_i$}
\label{fig:bigsurface}
\end{center}
\end{figure}

We will need the curves $\gamma_i, \, i \in \Z$ in the sequel: these are the curves corresponding to the separating edges of the cubic graph (they are indicated on Figure \ref{fig:bigsurface}). We order them by arbitrarily choosing a $\gamma_0$ and by asking that $\gamma_i$ separates $\gamma_{i-1}$ from $\gamma_{i+1}$ for all $i$. 

Although we have constructed $M$ explicitly, the arguments we use are clearly adaptable to other surfaces, including any bi-lipschitz equivalent surface. 

We now prove the result.

\begin{theorem}\label{thm:Z}
Any two elements of $\G_\infty(Z)$ are distance at most $3$ apart.
\end{theorem}

\begin{proof}
For $\mu,\nu \in \G_\infty(Z)$, we consider pants decompositions $v,w \in  \G_\infty(Z)$ such that $\mu\subset v$, $\nu\subset w$. We set 
$$L_v:= \sup\{ \ell(\alpha) \,|\, \alpha \in v \mbox{ is a connected component of }v\} 
$$
and similarly for $L_w$. We now set $L:=\max\{L_v,L_w\}$. 

Now for $x\in \{v,w\}$ and for any $\gamma_i$ defined as above, we consider the minimal finite subsurface $Z_{x,i}$ of $Z$ which contains all curves of $x$ that intersect $\gamma_i$ and the curve $\gamma_i$ itself. 

The surface $Z_{x,i}$ enjoys certain properties. As it contains $\gamma_i$, it separates the surface $Z$ and two of the connected components of $Z\setminus Z_{x,i}$ are infinite. As such the curves of $\partial Z_{x,i}$ also enjoy this property and note that they belong to the pants decomposition $x$.

It is of finite complexity bounded by a function of $L$. 
One way to see that the complexity is bounded is via the collar lemma: the number of intersection points between $x$ and $\gamma_i$ is bounded above by a function of $L$. It follows that the number of curves of $x$ that intersect $\gamma_i$ is bounded as well.

We now show that if $d_{Z}(\gamma_i, \gamma_j) > L$, then $Z_{x,i}$ and $Z_{y,j}$ for $x,y\in \{v,w\}$ are disjoint. Indeed, any curve contained inside $Z_{x,i}$ can be formed by arcs of curves of length at most $L$ and $\gamma_i$. As such they can be isotopically realized by curves that live in the subset of $Z$ consisting of points of distance at most $\frac{L}{2}$. By applying the same argument to $Z_{y,j}$ any curve of $Z_{x,i}$ can be realized disjointly from any curve of $Z_{y,j}$ and thus they do not intersect. 

We can now prove the main result. We consider a subset $\gamma_{i_k}, k \in \Z$ of the separating curves such that $d_Z(\gamma_{i_{k}}, \gamma_{i_{k+1}}) > L$. We also choose them so that 
$$
\sup_{k\in \Z}\{ d_Z(\gamma_{i_{k}}, \gamma_{i_{k+1}})\} < +\infty.
$$

For even $k\in \Z$ we consider the multicurve $v'\subset v$ obtained by considering the union of the (geodesic realizations of) the curves in $\cup_{k\in 2\Z} \partial Z_{v,i_{k}}$. By construction, $v'\in \G_\infty(Z)$ and $v$ and $v'$ span an edge (and so do $\mu$ and $v'$). We construct $w'$ in an analogous way by considering the curves in $\partial Z_{v,i_{k}}$ for odd $k$. By construction, the multicurves $v'$ and $w'$ span an edge. So we have a path $\mu, v', w', \nu$ and this completes the proof.
\end{proof}

\addcontentsline{toc}{section}{References}
\bibliographystyle{plain}

\end{document}